\documentclass[reqno]{tran-l}

\usepackage{amssymb,amsmath}
\usepackage[all]{xy}
\usepackage{graphicx}
\usepackage{accents}
\usepackage{xcolor}

\newtheorem{thm}{Theorem}

\newtheorem{lem}[thm]{Lemma}
\newtheorem{prop}[thm]{Proposition}
\theoremstyle{definition}
\newtheorem{defn}[thm]{Definition}
\theoremstyle{remark}

\begin{document}

\title[Gr\"{o}bner-Shirshov Basis for affine Weyl Group $\widetilde{A_n}$]
 {Gr\"{o}bner-Shirshov Basis and Reduced Words for affine Weyl Group $\widetilde{A_n}$}

\author{Erol Y{\i}lmaz, Cenap \"{O}zel and U\v{g}ur Ustao\v{g}lu }

\address{Department of Mathematics, Abant \.{I}zzet Baysal University, Bolu, Turkey}

\email{yilmaz\_e2@ibu.edu.tr, cenap@ibu.edu.tr, ugur1987@gmail.com}



\date{}

\dedicatory{}

\commby{}


\begin{abstract}
 Using Buchberger-Shirshov Algorithm and Composition-Diamond lemma we obtain the reduced Gr\"{o}bner-Shirshov bases of  $\widetilde{A_n}$ and classify all reduced words of the affine Weyl group  $\widetilde{A_n}$.
\end{abstract}

\maketitle

\section{Introduction}

 Gr\"{o}bner-Shirshov bases and normal form of the elements were already found for the Coxeter groups of type $A_n,B_n$ and $D_n$ in \cite {bokut1}. They also proposed a conjecture for the general form of Gr\"{o}bner-Shirshov bases for all Coxeter groups. In \cite {chen}, an example was given an to show that the
conjecture is not true in general. The Gr\"{o}bner-Shirshov bases of the other finite Coxeter groups are given in \cite {E1} and \cite {E2}. This paper is another example of finding Gr\"{o}bner-Shirshov bases for groups, defined by generators and defining relations.

We first cite some concepts and results from the literature which are related to the Gr\"{o}bner-Shirshov bases for the associative algebras. ( see \cite{bokut2,bokut3,shirshov} )

Suppose  $S$ is a linearly ordered set and $k$ is a field. Let $k\langle S \rangle$ be the free associative algebra over $k$ generated by $S$ and $S^*$ be the free monoid generated by $S$ where empty word is the identity which is denoted by $1$.

Let $S^*$ be equipped with a monomial ordering $<$. It means that $<$ is a well ordering that agrees with left and right multiplications by words:
$$
u>v \Rightarrow w_1uw_2 > w_1vw_2, \; \mathrm{for} \; \mathrm{all} \; w_1,w_2 \in S^*.
$$

A standard example of monomial ordering on $S^*$ is deg-lex ordering which first compare two words by length and then comparing them lexicographically where $S$ is a well-ordered set.

Let $f= \alpha \overline{f} +\sum \alpha_i u_i \in k \langle S \rangle$, where $\alpha, \alpha_i \in k, \overline{f} \in S^*$ and $u_i < \overline{f}$ for each $i$. Then we call $\overline{f}$ the leading word and $f$ monic if $\overline{f}$ has coefficient $1$. For a word $ w \in S^*$, we denote the length of $w$ by $|w|$.

\begin {defn}
For two monic polynomial $f$ and $g$ in $k\langle S\rangle $ and a word $w$, their composition defined by
$$
(f,g)_w=\left\{
  \begin{array}{ll}
    f-agb, & \hbox{if} \; w=\overline{f}=a\overline{g}b\\
    fb-ag, & \hbox{if} \;w=\overline{f}b=a\overline{g},|\overline{f}|+|\overline{g}|>|w|
  \end{array}
\right.
$$
The word $w$ is called the ambiguity of $f$ and $g$. The first type of composition is called the composition of including $g$ in $f$, and the second type is called the composition of intersection of $f$ and $g$. The transformation $f \mapsto f-agb$ is called the elimination of leading word (ELW) of g in f. Let $ R \subset k \langle S\rangle $ be a monic set. The composition $(f,g)_{w}$ is called reduced to $r$ relative to $R$ if $(f,g)_{w}=\sum \alpha_ia_ir_ib_i+r $ where every $\alpha_i \in k, a_i,b_i \in S^{*},r_i \in R$ with $ a_i \overline{r_i} b_i <w $ and the composition $(f,g)_{w}$ is called trivial relative to $R$ if $r=0$.

\end {defn}

The set $R \subset k\langle S\rangle$ is called Gr\"{o}bner-Shirshov bases if any composition of polynomials from $R$ is trivial relative to $R$.

The following lemma was first proved by Shirshov \cite {shirshov} for Lie algebras presented by generators and defining relations. He called it the composition lemma. Similar lemma for free associative algebras was formulated later by Bokut \cite{bokut1} and by Bergman \cite{bergman} under the name "Diamond lemma" after celebrated Newman's Diamond lemma \cite {newman} for graphs. This kind of lemmas are now named as composition-diamond lemmas. We will use Bokut's version of this lemma for free associative algebras. Similar ideas were independently discovered by Hironaka \cite{hironaka} for power series algebras and by  Buchberger \cite {buchberger1, buchberger2} for polynomial algebras.

\begin{lem} \text{(Composition-Diamond Lemma for associative algebras)}

Let $k$ be a field, $ A = k\langle S|R \rangle = k \langle S \rangle /Id(R)$ and $<$ a monomial ordering on
$S^*$, where $Id(R)$ is the ideal of $k\langle S\rangle$ generated by $R$. Then the following statements are
equivalent:

(i) $R$ is a Gr\"{o}bner-Shirshov basis.

(ii) $f \in Id(R) \Rightarrow \overline{f} = a\overline{s}b $ for some $s \in R$ and $a, b \in S^*$.

(iii) The set of $R$-reduced words $$ Red(R) = \{w \in S^*|w \neq a\overline{s}b, a, b \in S^*, s \in R \}$$ is a $k$-linear basis for the algebra
$A=k\langle S|R\rangle$.
\end{lem}

For commutative polynomials, this lemma is known as Buchberger's Theorem for the Gr\"{o}bner basis (see \cite{buchberger1,buchberger2}).

If $R \subset k \langle S \rangle$ is not a Gr\"{o}bner-Shirshov basis, then we reduce every nontrivial compositions to a polynomial relative to $R$ and add this polynomial to the $R$. Having repeated this procedure (possibly infinitely many times) we obtain a Gr\"{o}bner-Shirshov
basis $R^{comp}$. The process is called the Buchberger-Shirshov algorithm.

If the set $R$ consists of semigroup relations (i.e. $u - v$, where $u,v \in
S^*$), then each nontrivial composition of polynomials from $R$ has the same
semigroup form. Hence, $R^{comp}$ consists of semigroup relations too.

Let $A = smg\langle S|R \rangle$ be a semigroup presentation. Then $R \subset k \langle S \rangle$ and
we can obtain the Gr\"{o}bner-Shirshov bases $R^{comp}$. The set $R^{comp}$ does not
depend on the field $k$ and consists of semigroup relations. We will call $R^{comp}$
the Gr\"{o}bner-Shirshov bases for the semigroup $A$.

The main purpose of this paper is to find a Grobner-Shirshov basis and classify all reduced words for the affine Weyl group $\widetilde{A}_n$. The strategy for solving the problem is as follows: Let $R$ be the set of polynomials of the defining relations of $\widetilde{A}_n$. Using Buchberger-Shirshov algorithm we obtain new set $R'$ of polynomials including $R$. Then, by using the algorithm of elimination of leading words with respect to the polynomials in $R'$, all the words in the group $\widetilde{A}_n$ are reduced to the explicit classes of words. After that, we compute the number of the reduced words with respect to these classes by means of a generating function. This generating function turns out the be same  with the well known Poincar\'{e} polynomial of the affine Weyl group $\widetilde{A}_n$. Therefore, by the Composition-Diamond Lemma  the functions in $R'$ form Gr\"{o}bner-Shirshov basis for the affine Weyl group $\widetilde{A}_n$. Furthermore, one can easily see that this basis is in fact a reduced Gr\"{o}bner-Shirshov basis.

The results of this paper were obtained during M.Sc studies of  U\v{g}ur Ustao\v{g}lu at Abant \.{I}zzet Baysal University and are also contained in his thesis \cite{ugur}.

\section{Gr\"{o}bner-Shirshov Basis and Reduced Words}

\begin{defn} The affine Weyl group $\widetilde{A_n}$ has a presentation with generators $S=\{r_0,r_1, \ldots,r_n\}$ and defining relations

$ r_ir_i=1 \quad   0 \leq i \leq n$

$ r_ir_j=r_jr_i \quad   0\leq i < j-1 < n \; \mathrm{and} \; (i,j) \neq (0,n)$

$r_ir_{i+1}r_i=r_{i+1}r_ir_{i+1} \quad 0 \leq i \leq n-1$

$r_0r_nr_0=r_nr_0r_n$.

\end{defn}

Identifying each relation $u=v$ by a polynomial $u-v$, we define

$f_{1}^{(i)}=r_ir_i-1 \quad   0 \leq i \leq n$
\smallskip

$f_{2}^{(i,j)}=r_ir_j-r_jr_i \quad   0\leq i < j-1 < n \; \mathrm{and} \; (i,j) \neq (0,n)$
\smallskip

$f_{3}^{(i)}=r_ir_{i+1}r_i-r_{i+1}r_ir_{i+1} \quad 0 \leq i \leq n-1$
\smallskip

$f_4=r_0r_nr_0-r_nr_0r_n$

Let us define

$$
r_{ij}= \left\{
  \begin{array}{ll}
    r_ir_{i+1}\ldots r_j, & i<j; \\
    r_ir_{i-1}\ldots r_j, & i>j;\\
    r_i, & i=j; \\
    1, & i=1,j=0;\\
    1, & i=n,j=n+1.
  \end{array}
\right.
$$

\begin{lem} \label{gs}
Let $R=\{f_1,f_2,f_3,f_4\}$. A Grobner-Shirshov Basis of $\widetilde{A_n}$ with respect to deglex order with $r_0>r_1> \cdots >r_n$ contains the following polynomials.

$g_{1}^{(i,j)}=r_{ij}r_i-r_{i+1}r_{ij} \qquad 0 \leq i <j-1 < n $ with $(i,j)\neq (0,n) $
\smallskip

$g_{2}=r_{0n}r_0r_n-r_1r_{0n}r_0$
\smallskip

$g_{3}^{(j,k)}=r_0r_{nk}r_j-r_jr_0r_{nk} \qquad 2 \leq j<k-1<n $
\smallskip

$g_{4}^{(j)}=r_0r_{nj}r_{j+1}-r_jr_0r_{nj} \qquad 2 \leq j < n $
\smallskip

$g_{5}^{(k)}=r_0r_{nk}r_0-r_nr_0r_{nk} \qquad 2 \leq k <n $
\smallskip

$g_{6}^{(k,l)}=r_0r_{nk}r_{1l}r_{0l}-r_nr_0r_{nk}r_{1l}r_{0,l-1} \qquad 1\leq l <n, \;  2 \leq k\leq n$
\smallskip

$g_{7}^{(k,l)}=r_0r_{nk}r_{1l}r_0r_{nk}-r_1r_0r_{nk}r_{1l}r_0r_{n,k+1} \qquad 1 \leq l< k-1 <n $
\smallskip

$g_{8}^{(k,l)}=r_0r_{nk}r_{1l}r_0r_{n,k-1}-r_1r_0r_{nk}r_{1l}r_0r_{nk} \quad 3 \leq k\leq n, \quad k-1\leq l\leq n $
\smallskip

$g_{9}^{(j,k,l)}=r_0r_{nk}r_{1l}r_0r_{nj}r_{1l}-r_nr_0r_{nk}r_{1l}r_0r_{nj}r_{1,l-1}$ \\
\hspace*{30pt}$\qquad2\leq k\leq n-1, \quad k+1\leq j\leq n,\quad 1\leq l\leq j-2$
\smallskip

$g_{10}^{(j,k,l)}=r_0r_{nk}r_{1l}r_0r_{nj}r_{1,l+1}-r_nr_0r_{nk}r_{1l}r_0r_{nj}r_{1l}$ \\
\hspace*{30pt}$\qquad 2\leq k\leq n,\quad k\leq j\leq n,\quad j-1\leq l\leq n-1$
\smallskip

\end{lem}

\begin{proof}
Let $R=\{f_1,f_2,f_3,f_4\}$. We apply the Buchberger-Shirshov algorithm to the $R$. We show every ELW in the below computations except ELW's of $f_2{(i,j)}$, the cummutativity relations. Notice that at this point we are not claiming that this is a Gr\"{o}bner-Shirshov basis for $\widetilde{A}_n$.

\smallskip

\noindent $(f_{3}^{(i)},f_{2}^{(i,i+2)})=g_{1}^{(i,i+2)}.$

\smallskip

\noindent $(g_{1}^{(i,j-1)},f_{2}^{(i,j)})=g_{1}^{(i,j)} \quad \mathrm{for} \quad  j=i+3,\ldots,n.$

\smallskip

\noindent $(g_{1}^{(0,n-1)},f_{4})=g_{2}.$

\smallskip

\noindent $(f_{2}^{(0,j)},f_{2}^{(j,n)})=g_{3}^{(j,n)}.$

\smallskip

\noindent $(g_{3}^{(j,k+1)},f_{2}^{(j,k)})=g_{3}^{(j,k)}\quad \mathrm{for} \quad k=n-1,\ldots,j+2.$

\smallskip

\noindent $(f_{2}^{(0,n-1)},f_{3}^{(n-1)})=g_{4}^{(n-1)}.$

\smallskip

\noindent $(g_{3}^{(j,j+2)},f_{3}^{(j)})=g_{4}^{(j)} \quad  \mathrm{for} \quad j=n-2,...,2.$

\smallskip

\noindent $(f_{4},f_{2}^{(0,n-1)})=g_{5}^{(n-1)}.$

\smallskip

\noindent $(g_{5}^{(k+1)},f_{2}^{(0,k)})=g_{5}^{(k)} \quad  \mathrm{for} \quad k=n-2,\ldots,2 .$

\smallskip

\noindent $(f_{4},f_{3}^{(0)})=g_{6}^{(n,1)}$

\smallskip

\noindent $(g_{5}^{(k)},f_{3}^{(0)})=g_{6}^{(k,1)}\quad \mathrm{for} \quad k=n-1,\ldots,2.$

\smallskip

\noindent $(g_{6}^{(k,l-1)},f_{3}^{(l-1)})=g_{6}^{(k,l)}\quad \mathrm{for} \quad l=2,\ldots,n-1 .$

\smallskip

\noindent $(f_{3}^{(0)},f_4)=r_0r_nr_1r_0r_n-r_1r_0r_nr_1r_0=g_{7}^{(n,1)}.$

\smallskip

\noindent $(g_{1}^{(0,l)},f_{4})=g_{7}^{(n,l)} \quad \mathrm{for} \quad l=2,\ldots,k-2.$

\smallskip

\begin{flalign*}
(g_{7}^{(k+1,l)},g_{3}^{(k+1,k)})& = r_{0}r_{n,k+1}r_{1l}r_{k}r_{0}r_{nk}-r_{1}r_{0}r_{n,k+1}r_{1l}r_{0}r_{n,k+2}r_{k}r_{k+1}  \\
   &= g_{7}^{(k,l)}-r_{1}r_{0}r_{n,k+1}r_{1l}g_{3}^{(k,k+2)}r_{k+1} \quad \mathrm{for} \quad k=n-2,\ldots,3.&
\end{flalign*}

\smallskip

\begin{flalign*}
(g_{7}^{(n,l)},g_{3}^{(n,n-1)})& = r_{0}r_{n}r_{1l}r_{n}r_{0}r_{n,n-1}-r_{1}r_{0}r_{n}r_{1l}r_{0}r_{n,n-1}r_{n}  \\
   &= g_{7}^{(n-1,l)}-r_{1}r_{0}r_{n}r_{1l}f_{2}^{(0,n-1)}r_{n} \quad \mathrm{for} \quad k=n-1.&
\end{flalign*}

\smallskip

\begin{flalign*}
(g_{7}^{(k,k-2)},g_{4}^{(k-1)})&=r_{0}r_{nk}r_{1,k-1}r_{0}r_{n,k-1}-r_{1}r_{0}r_{nk}r_{1,k-2}r_{0}r_{n,k+1}r_{k-1}r_{k}\\
   &=g_{8}^{(k,k-1)}-r_{1}r_{0}r_{nk}r_{1,k-2}g_{3}^{(k-1,k+1)}r_{n}\quad \mathrm{for} \quad k=n-1,\ldots,3.&
\end{flalign*}

\smallskip

\begin{flalign*}
(g_{7}^{(n,n-2)},g_{4}^{(n-1)})&=r_{0}r_{n}r_{1,n-1}r_{0}r_{n,n-1}-r_{1}r_{0}r_{n}r_{1,n-2}r_{0}r_{n-1}r_{n}\\
   &=g_{8}^{(n,n-1)}-r_{1}r_{0}r_{n}r_{1,n-2}f_{2}^{(0,n-1)}r_{n}\quad \mathrm{for} \quad k=n.&
\end{flalign*}

\smallskip

\noindent $(g_{2},g_{4}^{(n-1)})=r_{0,n-2}f_{3}^{(n-1)}r_{0}r_{n}r_{n-1}-r_{1}r_{0n}f_{2}^{(0,n-1)}r_{n}-r_{1}r_{0,n-2}f_{3}^{(n-1)}r_{0}r_{n}+g_8^{(n,n)}$

\smallskip

\begin{flalign*}
(g_{8}^{(k+1,l)},g_{4}^{(k-1)})&=r_{0}r_{n,k+1}r_{1,k-2}g_{1}^{(k-1,l)}r_{0}r_{n,k-1}-r_{1}r_{0}r_{n,k+1}r_{1l}g_{3}^{(k-1,k+1)}r_{k} \\
&-r_{1}r_{0}r_{n,k+1}r_{1,k-2}g_{1}^{(k-1,l)}r_{0}r_{n,k+1}r_{k}+g_{8}^{(k,l)}\quad \mathrm{for} \quad l=k,\ldots,n .&
\end{flalign*}

\smallskip

\noindent $(g_{6}^{(k,1)},f_{2}^{(1,n)})=g_{9}^{(k,n,1)}.$

\smallskip

\noindent $(g_{9}^{(k,j+1,1)},f_{2}^{(1,j)})=g_{9}^{(k,j,1)}\quad \mathrm{for} \quad j=n-1,\ldots,k+1.$

\smallskip

\begin{flalign*}(g_{9}^{(k,j,l-1)},f_{3}^{(l-1)})&=r_0r_{nk}r_{1,l-1}r_0r_{nj}r_{1,l-2}r_{l}r_{l-1}r_l-r_nr_0r_{nk}r_{1,l-1}r_0r_{nj}r_{1,l-2}r_{l}r_{l-1}\\
&=g_{9}^{(k,j,l)} \quad \mathrm{for} \quad l=2,\ldots,j-2.&
\end{flalign*}

\smallskip

\begin{flalign*}(g_{9}^{(k,j+1,j-1)},f_{3}^{(j-1)})=&r_0r_{nk}r_{1,j-1}r_0r_{n,j+1}r_{1,j-2}r_{j}r_{j-1}r_{j}\\
&\qquad \quad -r_{n}r_{0}r_{nk}r_{1,j-1}r_{0}r_{n,j+1}r_{1,j-2}r_{j}r_{j-1}\\
&=g_{10}^{(k,j,j-1)}\quad \mathrm{for} \quad j=n-1,\ldots,k.& \\
\end{flalign*}

\smallskip

\begin{flalign*}
(g_{6}^{(k,n-1)},f_{3}^{(n-1)})=& r_0r_{nk}r_{1,n-1}r_{0,n-2}r_{n}r_{n-1}r_{n}\\
&\qquad -r_{n}r_0r_{nk}r_{1,n-1}r_{0}r_{n}r_{1,n-1} = g_{10}^{(k,n,n-1)} .&
\end{flalign*}

\smallskip

\begin{flalign*}(g_{10}^{(k,j,l-1)},f_{3}^{(l)})=&r_0r_{nk}r_{1,l-1}r_0r_{nj}r_{1,l-1}r_{l+1}r_{l}r_{l+1}\\
& \qquad -r_nr_0r_{nk}r_{1,l-1}r_0r_{nj}r_{1,l-1}r_{l}r_{l+1}r_l\\
=&r_{0}r_{nk}r_{1,l-1}g_{4}^{(l)}r_{l-1,j}r_{1,l+1}-r_{n}r_{0}r_{nk}r_{1,l-1}g_{4}^{(l)}r_{l-1,j}r_{1l}\\
&\qquad +g_{10}^{(k,j,l)} \quad \mathrm{for} \quad l=j-1,\ldots,n-1.&
\end{flalign*}

In the above equation if $l=j$, then $r_{l-1,j}$ assumed to be the identity $1$.

\end{proof}

Let $R'= R \cup \{g_1, \ldots, g_{10} \}.$ We want to find the properties of the elements of the set $Red(R') = \{w \in S^*|w \neq a\overline{f}b, a, b \in S^*, f \in R'\}$. If $w \in Red(R')$, then we call it a reduced word.

Notice that elements of  $R'$ not containing $r_0$ is in fact a Gr\"{o}bner-Shirshov basis for the Coxeter group $A_n$. The following lemma is just another way of expressing of the Lemma 3.2 of \cite{bokut1}.
\begin{lem}\label{ri}
Any reduced word not containing $r_{0}$ is in the form $$r=(r_{nj_{n}})^{\alpha_{n}}(r_{n-1,j_{n-1}})^{\alpha_{n-1}},\ldots,(r_{2j_{2}})^{\alpha_{2}},(r_{1j_{1}})^{\alpha_{1}}$$ where $i\leq j_{i}\leq n$ and $\alpha_{i}\in \{0,1\}$.
\end{lem}

After investigation of leading words of the elements of $G'$, we can claim the following results. For convenience  we write $r_{0}r_{n,n+1}r_{1l}$ instead of $r_{0l}$ and  $r_{0}r_{nk}r_{10}$ instead of $r_{0}r_{nk}$.

\begin{lem}\label {r0}
The following words are reduced.

  (i) $w=r_{0}r_{nk}r_{1l}$ $2\leq k\leq n+1$, $0\leq l\leq n$

  (ii) $$(r_0r_{nk}r_{1l})(r_0r_{np}r_{1q})=\left\{
                                            \begin{array}{ll}
                                              (k<p) \wedge (l>q), & \hbox{if} \; \; \; q-p <l-k<-1\\
                                              (k \leq p) \wedge (l>q), & \hbox{if} \; \; \; (q-p<-1) \wedge (l-k \geq -1)\\
                                              (k \leq p) \wedge (l \geq q), & \hbox{if} \; \; \;  l-k>q-p \geq -1
                                            \end{array}
                                          \right.$$

\end{lem}

\begin{lem} \label {rir0}
Let $w$ be a reduced word starting with $r_i$ for $i=1,\ldots,n$ and let $t$ be a reduced word starting with $r_0$. Then $wt$ is also a reduced word.
\end{lem}

\begin{proof}
 The results follows from the following observation. Any leading word starting with $r_i$ for $i=1,\ldots,n$ in $G'$ do not contains $r_0$.
\end{proof}

\begin{lem} \label{lem8}
Let $w_{1},w_{2},\ldots w_{k},w_{k+1}$ be reduced words in the one of the forms given in the first three items of Lemma ~\ref{r0}. If $w_{1}w_{2}\ldots w_{k}$ and $w_{2}w_{3}\ldots w_{k+1}$ are reduced, then $w_{1}w_{2}\ldots w_{k+1}$ is also reduced.
\end{lem}
\begin{proof}
Let $w=w_{1}w_{2}\ldots w_{k+1}$. The only possible reduced subword of $w$ is in the form $rw_{2}w_{3}\ldots w_{k}s$ where $r$ and $s$ are subwords of $w_{1}$ and $w_{k+1}$, respectively. Since $w_{2}w_{3}\ldots w_{k}w_{k+1}$ is reduced word, $w_{2}w_{3}\ldots w_{k}s$ is also reduced word starting with $r_{0}$. Since $r$ is a subword of $w_{1}$, $r$ is also a reduced word not containing $r_{0}$. By Lemma ~\ref{rir0}, $rw_{2}w_{3}\ldots w_{k}s$ is also reduced. Hence $w_{1}w_{2}\ldots w_{k+1}$ must be a reduced word.
\end{proof}

\begin{defn}\label{u}
Let $a_{i}=r_{0}r_{nk}r_{1l}$ for $2\leq k\leq n+1$, $0\leq l\leq n$ and $l-k=i-2$.

Let $$u=(a_{n})^{m_{n}}(a_{n-1})^{m_{n-1}}\ldots (a_{1})^{m_{1}}$$ where $m_{i}\geq 0$ for $i=1,\ldots,n$. Furthermore, if $a_{i}=r_{0}r_{nk}r_{1l}$, then $a_{i+1}=r_{0}r_{n,k+1}r_{1l}$ or $a_{i-1}=r_{0}r_{nk}r_{1,l-1}$. Notice that the number of possible $u$'s is $2^{n-1}$.
We call $a_i$'s the components of the word $u$.
\end{defn}

\begin{defn}
Let $b_t=r_{0}r_{np}r_{1q}$ for $2\leq p\leq n+1$ and $0\leq q\leq n$ satisfying  $q-p<-1$. There are $n!$ such words.

Let $$v^{(b_t)}=b_t (b_{t-1})^{\alpha_{t-1}}\cdots (b_{s})^{\alpha_s}$$ where $\alpha_{i}\in \{0,1\}$. Furthermore, if $b_{i}=r_{0}r_{np_i}r_{1q_i}$, then $b_{i-1}=r_{0}r_{np_{i-1}}r_{1q_{i-1}}$ for $p_i<p_{i-1}$ and $q_i>q_{i-1}$. If $p_i=n+1$ or $q_i=0$, then $\alpha_{j}=0$ for $j=s,\ldots,i-1$.

For the convenience, we define $1=r_0r_{n,\infty}r_{1,-1}$ and $v{(1)}=1$.
\end{defn}

\begin{prop}{\label w}
Let $u$ and $v^{(b_t)}$ be words defined above where $a_1=r_0r_{nk}r_{1l}$ and $b_t=r_0r_{np}r_{1q}$. Then the words $u$, $v^{(b_t)}$ and $w=u v^{(b_t)}$ are reduced if $p \geq k$ and $q<l$ in $w$.
\end{prop}

\begin{proof}
The result is easily follows from Lemma ~\ref{r0} and Lemma \ref{lem8}.
\end{proof}

Figure 1.1 shows every possible reduced words for $n=4$.

Let $$w_1=(r_0r_{42}r_{14})^{m_1}(r_0r_{42}r_{13})^{m_2}(r_0r_{43}r_{13})^{m_3}(r_0r_{4}r_{13})^{m_4}r_{01}$$ and $$w_2=(r_0r_{42}r_{14})^{m_1}(r_0r_{43}r_{14})^{m_2}(r_0r_{4}r_{14})^{m_3}(r_0r_{4}r_{13})^{m_4}r_{01}.$$ Then $w_1$ and $w_2$  are two reduced words in $\widetilde{A}_4$. If we take $m_2=m_3=0$ and $m_1=m_4=1$, then the subword $(r_0r_{42}r_{14})(r_0r_{4}r_{13})r_{01}$ is written twice. To avoid this situation, we define the arranged words.

\begin{figure}
  \includegraphics[width=350pt]{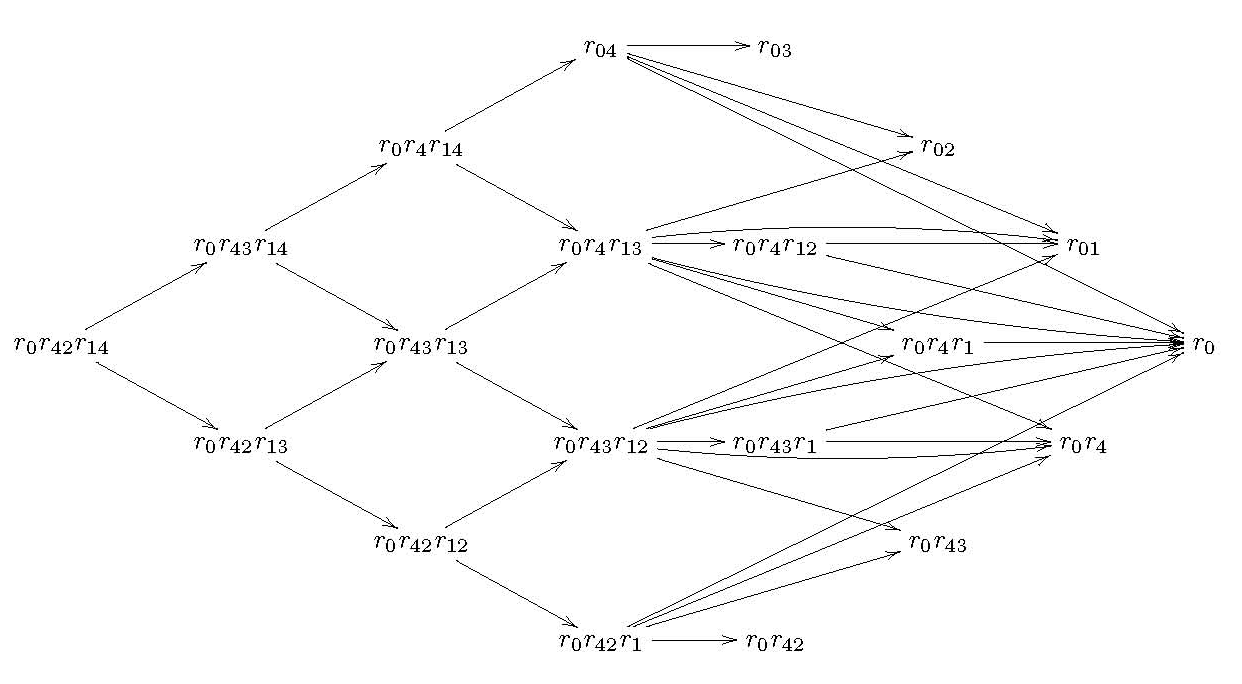}\\
  \caption{1}\label{1}
\end{figure}

\begin{defn}\label{arranged}
 Let $w=u v^{(b_t)}$ be a reduced word where $b_{t}=r_{0}r_{np}r_{1q}$ and
 $$
 u=(a_n)^{(m_n)}(a_{n-1})^{(m_{n-1})} \cdots (a_1)^{(m_1)}.
 $$

 For $i=2,\ldots,n-1$ , let $m_i \geq 1$ if $a_{i+1}=r_{0}r_{nk}r_{1,l+1}$, $a_{i}=r_{0}r_{nk}r_{1l}$ and $a_{i-1}=r_{0}r_{n,k+1}r_{1l}$.  If $a_2=r_{0}r_{nk}r_{1,l+1}, a_1=r_{0}r_{nk}r_{1l}$ and $p>k$, then let $m_1 \geq 1$. Then $w$ is called an arranged word and the components $a_i$ where $m_i \geq 1$ are called a marked component of $w$.
\end{defn}

\begin{thm}
If all reduced words $w=u v^{(b_t)}$ are arranged, then each subword is written uniquely.
\end{thm}
\begin{proof}
Let $w_1=u_1 v^{(b_t)}$ and $w_2=u_2v^{(b_t)}$ be two arranged word. Since the only possibility for $a_n=r_0r_{n2}r_{1n}$, the first components of them are the same. Let us assume that $a_n,a_{n-1},\ldots,a_j$ are common in $w_1$ and $w_2$. If the exponents of $a_i$ for $i=j,\ldots,n$ were $1$ and the others were $0$
in both $w_1$ and $w_2$, then the word $a_na_{n-1}\cdots a_j$ would be written twice.

Let us assume $a_j=r_0r_{nk}r_{1l}$. Moreover, let  $a_{j-1}=r_0r_{n,k+1}r_{1l}$ in $w_1$ and $a_{j-1}=r_0r_{nk}r_{1,l-1}$ in $w_2$. Then there exits $s \geq l-1$ the components of $w_2$ between $j-1$ and $s+1$ are
 $$
 (r_{0}r_{nk}r_{1,l-1})(r_{0}r_{nk}r_{1,l-2})\cdots (r_{0}r_{nk}r_{1s})(r_{0}r_{n,k+1}r_{1s})
 $$
 when $s-k>-1$. If $s-k=-1$, we have
 $$
 (r_{0}r_{nk}r_{1,l-1})(r_{0}r_{nk}r_{1,l-2})\cdots (r_{0}r_{nk}r_{1s})(r_{0}r_{np}r_{1q})
 $$

 Therefore $r_0r_{nk}r_{1s}$ is a marked component in the first case. By Definition \ref{u}, $r_0r_{nk}r_{1l}$ can not be a component of $w_1$. The component of $w_1$ in the same position can be $r_0r_{n\overline{k}}r_{1 \overline{s}}$ where $ \overline{k}> k$ and $\overline{s}<s$. Since $b_t$ is common in both words, $p \geq \overline{k} >k$. Therefore $r_0r_{nk}r_{1s}$ is  also a marked component in the second case.Hence the word $(a_n)a_{n-1}\cdots (a_j)$ can be written only in $w_1$ not in $w_2$.

\end{proof}

\section{Counting Reduced Words}

\begin{thm}\label{reduced}
Let $\widehat{w}=a_{i_{s}}a_{i_{s-1}}\ldots a_{i_{1}}v^{(b_t)}$ where $a_{i_{j}}=r_{0}r_{nk_{j}}r_{nl_{j}}$.Then $a_{i_{j}}$'s are marked components of an arranged word $w=uv^{(b_t)}$ where $b_t=r_0r_{np}r_{1q}$ if and only if $k_{s}<k_{s-1}<\ldots<k_{1}<p$ and $q<l_{1}<l_{2}<\ldots<l_{s}<n$
\end{thm}

\begin{proof}
By construction of arranged words, the marked components satisfy the given conditions.
Conversely, if $\widehat{w}=a_{i_{s}}a_{i_{s-1}}\ldots a_{i_{1}}v^{(b_t)}$ is a word satisfying the conditions, then an arranged word $(w=u v^{(b_{t})})$ whose marked components are $a_{i_j}$'s can be obtained as follows:

The components up to $r_0r_{nk_s}r_{1l_s}$ are $$(r_{0}r_{n2}r_{1n})\cdots(r_{0}r_{nk_{s}}r_{1n})(r_{0}r_{nk_{s}}r_{1,n-1})\cdots(r_{0}r_{nk_{s}}r_{1l_{s}}),$$
the components between $(r_{0}r_{nk_{j}}r_{1l_{j}})$ and $(r_{0}r_{nk_{j-1}}r_{1l_{j-1}})$ are
 $$(r_{0}r_{nk_{j}}r_{1l_{j}})\cdots(r_{0}r_{nk_{j-1}r_{1l_{j}}})(r_{0}r_{nk_{j-1}}r_{1,l_{j}-1})\cdots(r_{0}r_{nk_{j-1}}r_{1l_{j-1}})$$ and the last part of the word$w=u v^{(b_{t})}$ is
$$(r_{0}r_{nk_{2}}r_{1l_{2}})\cdots(r_{0}r_{nk_{1}r_{1l_{2}}})(r_{0}r_{nk_{1}}r_{1,l_{2}-1})\cdots(r_{0}r_{nk_{1}}r_{1l_{1}})(v^{(b_{t})}).$$
\end{proof}

Therefore the number of the elements in an arranged word $w=u v^{(b_t)}$ given by the generating function
$$
\frac{x^\alpha}{(1-x^{2n})(1-x^{2n-1})\cdots (1-x^{n+1})}
$$
where $\alpha$ is the length of the word $\widehat{w}$. In order to count all reduced words starting with $r_0$ we have to find the number words $\widehat{w}$ whose length is $\alpha$ for any power $ \alpha$. To do this, we will find a correspondence between these words and some special partitions of integers.

If $m$ is a positive integer, then a partition of $m$ is a nonincreasing sequence of positive integers $p_{1},p_{2},\ldots,p_{k}$ whose sum is $m$. Each $p_{i}$ is called a part of the partition. Let $n$ be a positive integer. Any partition $m=d_{1}+d_{2}+\ldots+d_{k}$ where $k\leq n$ can be identify by the $n-$tuple $(d_{1},d_{2},\ldots,d_{k},0,0,\ldots,0)$.

We can also represent each word $r_0r_{nk}r_{1l}$ with the n-tuple $(k,1,\ldots,1,0,\ldots,0)$, $r_{nk}$ with  $(k,\ldots,0)$ and $r_{0l}$ with $(1,1\ldots,1,0,\ldots,0)$ where number one 1's is equal to $l$ for $2 \leq k \leq n-1, 1 \leq l \leq n$.

\begin{defn} Let $n$ be a positive integer. The $n$-tuples $(k,1\ldots,1,0,\ldots,0)$ where number of 1's is $l$ for $1\leq k\leq n$, $1\leq l\leq n-1$ are called basic partitions. The basic partition $(k_{1},1\ldots,1,0\ldots,0)$ is said to be connected to the basic partition $(k_{2},1,\ldots,1,0,\ldots,0)$ if $k_{1}>k_{2}$ and the number of 1's in the first one is greater than number of 1's in the second one. Hence a sequence of connected partition $a_{1},a_{2},\ldots,a_{m}$ corresponds to a word $\widehat{w}$ given in Theorem \ref{reduced}.

\end{defn}

\begin{thm}\label {qbinomial}
There is one to one correspondence between words $\widehat{w}$ and the partitions in which there are at most $n$ parts and in which no parts is larger than $n$.
\end{thm}
\begin{proof}
Since we identify each word $\widehat{w}$ with a sequence of connected basic partitions, we must find a correspondence between sequences of connected partitions and the partitions fit into a box of size $n \times n$. Let $a_{1},a_{2},\ldots,a_{m}$ be a sequence of connected partitions where $a_{i}=(k_{i},\underaccent{l_{i}}{\underbrace{1,\ldots,1}}0,\ldots,0)$. Hence $k_{i}>k_{j}$ and $l_{i}>l_{j}$ for $1 \leq i < j\leq m$.

Define
$$\bigoplus_{i=1}^{m}a_{i}=\sum_{i=1}^{m}\sigma_{1}^{i-1}(a_{i})$$

where $\sigma(p_{1},p_{2},\ldots,p_{n-1},p_{n})=(p_{n},p_{1},p_{2},\ldots,p_{n-1})$. Then

$$ \bigoplus_{i=1}^{m}a_{i}=(k_{1},k_{2}+1,\ldots,k_{m}+m-1,\underaccent{l_{m}}{\underbrace{m,\ldots,m}},
 \underaccent{l_{m-1}-(l_{m}+1)}{\underbrace{m-1,\ldots,m-1}},\ldots,\underaccent{l_{1}-(l_{2}+1)}{\underbrace{1,\ldots,1}},0\ldots,0)$$

We prove the last equation by induction on $m$.

Let $m=2$, Since $l_1 -1 \geq l_2$,

\begin{eqnarray*}
  a_{1}\oplus a_{2} &=& (k_{1},\underaccent{l_{1}}{\underbrace{1,\ldots,1}},0,\ldots,0)+(0,k_{2},\underaccent{l_{2}}{\underbrace{1,\ldots,1}},0,\ldots,0)  \\
 &=& (k_{1},k_{2}+1,\underaccent{l_{2}}{\underbrace{2,\ldots,2}},\underaccent{l_{1}-(l_{2}+1)}{\underbrace{1,\ldots,1}},0,\ldots,0)
\end{eqnarray*}

Let us assume that $$\bigoplus_{i=1}^{m-1}a_{i}=(k_{1},k_{2}+1,\ldots,k_{m-1}+m-2,\underaccent{l_{m-1}}{\underbrace{m-1,\ldots,m-1}},\ldots,\underaccent{l_{1}-(l_{2}+1)}{\underbrace{1,\ldots,1}},0,\ldots,0).$$

\begin{eqnarray*}
  \bigoplus_{i=1}^{m}a_{i} &=& \sum_{i=1}^{m}\sigma^{i-1}(a_{i}) \\
   &=& \sum_{i=1}^{m-1}\sigma^{i-1}(a_{i})+\sigma^{m-1}(a_{m})  \\
   &=& \bigoplus_{i=1}^{m-1}a_{i}+\sigma^{m-1}(a_{m}) \\
   &=& (k_{1},k_{2}+1,\ldots,k_{m-1}+m-2,\underaccent{l_{m-1}}{\underbrace{m-1,\ldots,m-1}},\ldots,\underaccent{l_{1}}{\underbrace{1,\ldots,1}},0,\ldots,0) \\
   &&  + \: (\underaccent{m-1}{\underbrace{0,0,\ldots,0}},k_{m},\underaccent{l_{m}}{\underbrace{1,\ldots,1}},0,\ldots,0) \\
   &=& (k_{1},k_{2}+1,\ldots,k_{m}+m-1,\underaccent{l_{m}}{\underbrace{m,\ldots,m}},\ldots,\underaccent{l_{1}-(l_{2}+1)}{\underbrace{1,\ldots,1}},0,\ldots,0)
\end{eqnarray*}

The last equality easily follows from the fact $l_{m-1}-1 \geq l_{m}$. Since $n \geq k_1 \geq k_2+1 \geq \cdots > k_{m}+(m-1) \geq m$, the last line corresponds to a partition of $2n$ fits into a $n$ by $n$ box.

Conversely let $m=(m_1,m_2,\ldots,m_n)$ be a partition where $n \geq m_1 \geq m_2 \geq \cdots \geq m_n \geq 0$. If $i_1$ is the last index such that $m_{i_1} \neq 0$, then let $a_1=(m_1,1,\ldots,1,0,\ldots0)$ where the last $1$ in $i_1$-th position. Then let $$x=\sigma^{-1}(m-a_1)=(m_2-1,\ldots,m_{i_1}-1, 0,\ldots,0)$$ and $a_2=(m_2-1,\ldots,1,0,\ldots,0)$ where the position of the last nonzero element in $x$ and the position of last $1$ in $a_2$ are same. Clearly, $a_1$ and $a_2$ are basic partitions and $a_1$ is connected to $a_2$. Continuing the process until reaching the $(0,\ldots,0)$, one can obtain a sequence of connected basic partitions.
\end{proof}

\begin{defn}

Let be positive integers. The q-binomial is defined by
$$
\left(
  \begin{array}{c}
    m \\
    r \\
  \end{array}
\right)_q=\frac{(1-q^m)(1-q^{m-1}) \cdots (1-q^{m-r+1})}{(1-q)(1-q^2) \cdots (1-q^r)}
$$
\end{defn}

Although the formula in the first clause appears to involve a rational function, it actually designates a polynomial, because the division is exact in $\mathbb{Z}[q]$.
A standard combinatorial interpretation for q-binomial is that it counts the number of partitions that will fit into a box of size $ k \times (n - k)$, weighted by the size of the partition. In particular the q-binomial
$$
\left(
  \begin{array}{c}
    2n \\
    n \\
  \end{array}
\right)_x
=\frac{(1-x^{2n})(1-x^{2n-1}) \cdots (1-x^{n+1})}{(1-x)(1-x^2) \cdots (1-x^{n})}$$
counts the number of partitions in which there are at most $n$ parts and in which no parts is larger than $n$.

Now, we can proof the main result of this paper.

\begin{thm}
  Then the reduced Gr\"{o}bner-Shirshov basis of the affine Weyl group $\widetilde{A}_n$ is the set $G'$. Moreover all the reduced words are the form $rw$ where $r$ is a reduced word not including $r_0$  and $w$ is a arranged word.
\end{thm}
The reduced words not including $r_0$ is given in Lemma \ref{ri}. It is easy to see that the number of such words given by the generating function
$$
(1+x)(1+x+x^2)\cdots(1+x+ \cdots +x^n).
$$
Theorem \ref{reduced} and Theorem \ref{qbinomial} imply that the number of arranged words given by the generating function

\begin{eqnarray*}
  \frac{\left(
  \begin{array}{c}
    2n \\
    n \\
  \end{array}
\right)_x}{(1-x^{2n})(1-x^{2n-1})\cdots (1-x^{n+1})} &=& \frac{\frac{(1-x^{2n})(1-x^{2n-1}) \cdots (1-x^{n+1})}{(1-x)(1-x^2) \cdots (1-x^{n})}}{(1-x^{2n})(1-x^{2n-1})\cdots (1-x^{n+1})} \\
   &=& \frac{1}{(1-x)(1-x^2)\cdots(1-x^n)}.
\end{eqnarray*}

By Lemma \ref{rir0}, the reduced words of $\widetilde{A}_n$ are in the form $rw$. Hence the number of reduced words given by the generating function
$$
\frac{(1+x)(1+x+x^2)\cdots(1+x+ \cdots +x^n)}{(1-x)(1-x^2)\cdots(1-x^n)}
$$
which is well known Poincar\'{e} polynomial of the affine Weyl group  $\widetilde{A}_n$. (see \cite{humphreys}). Therefore these are all reduced words of $\widetilde{A}_n$.
Hence by Composition-Diamond Lemma, $G'$ is a Gr\"{o}bner-Shirshov basis of  $\widetilde{A}_n$. In fact, $G'$ is a reduced  Gr\"{o}bner-Shirshov basis.


\end{document}